\documentclass{article}
\usepackage{amsmath,amssymb,amsthm}

\title{The Consistency of Arithmetic}

\author{Timothy Y. Chow}

\date{July 11, 2018}

\newtheorem{theorem}{Theorem}

\begin{document}
\maketitle

In 2010,
Vladimir Voevodsky, a Fields Medalist and a
professor at the Institute for Advanced Study, gave a lecture
entitled, ``What If Current Foundations of Mathematics Are Inconsistent?''
Voevodsky invited the audience to consider seriously the possibility
that first-order Peano arithmetic (or PA for short) was inconsistent.
He briefly discussed two of the standard proofs of the consistency
of~PA (about which we will say more later), and explained why
he did not find either of them convincing.  He then said that he
was seriously suspicious that an inconsistency in~PA might someday
be found.

About one year later, Voevodsky might have felt vindicated when
Edward Nelson, a professor of mathematics at Princeton University,
announced that he had a proof not only that PA was inconsistent,
but that a small fragment of primitive recursive arithmetic
(PRA)---a system that is widely
regarded as implementing a very modest and ``safe''
subset of mathematical reasoning---was inconsistent~\cite{nelsonpra}.
However, a fatal error in his proof was soon detected by Daniel Tausk
and (independently) Terence Tao.  Nelson withdrew his claim, remarking
that the consistency of~PA remained ``an open problem.''

For mathematicians without much training in formal logic, 
these claims by Voevodsky and Nelson may seem bewildering.
While the consistency of
some axioms of \emph{infinite set theory} might be debatable,
is the consistency of~PA really ``an open problem,'' as Nelson claimed?
Are the existing proofs of the consistency of~PA suspect,
as Voevodsky claimed?  If so, does this mean that we cannot be sure
that even basic mathematical reasoning is consistent?

This article is an expanded version of an answer that I posted
on the MathOverflow website in response to the question,
``Is PA consistent? do we know it?''
Since the question of the consistency of PA seems to come up
repeatedly, and continues to generate confusion,
a more extended discussion seems worthwhile.

\section{Some Preliminaries}

One of the great achievements of the late 19th and
early 20th centuries was the recognition that many
seemingly \emph{metamathematical} questions---i.e.,
questions about the mathematical enterprise as a whole,
such as the validity of its methods of reasoning---could be
formulated as \emph{mathematical} questions,
and therefore studied mathematically.
In particular, the consistency of~PA can be thought of as
a purely mathematical assertion, and so one can ask the
usual questions that one typically asks of mathematical
statements---has it been proved?  and if so, what does the proof look like?

To understand the status of the statement ``PA is consistent,''
we must therefore first familiarize ourselves with
the relevant mathematical results.
Below, we review the main proofs of the consistency of~PA,
and then discuss their implications.

There is one sociological fact that contributes to the confusion
surrounding the consistency of~PA, namely that even though mathematicians
will agree in principle that every proof must start with some axioms,
in practice they almost never state explicitly what axioms they are
assuming.  If pressed, most mathematicians will usually say that the
generally accepted axiomatic system for mathematics is ZFC,
the \emph{Zermelo--Fraenkel axioms} for set theory
together with the Axiom of Choice.
Ironically, most mathematicians cannot even \emph{state} the axioms
of~ZFC precisely, let alone explicitly verify that their proofs
can be formalized in~ZFC.  Nevertheless, for most mathematicians,
ZFC acts as the \emph{de jure} foundation for mathematics,
and if someone does not bother to state explicitly
what axioms they are ultimately relying on,
then we can usually assume that ZFC will suffice.

\section{A ZFC Proof that PA is Consistent}

If the consistency of~PA is a mathematical question,
and ZFC is supposed to be the foundation for mathematics,
then a natural first question to ask is, is the consistency of~PA
provable in~ZFC?  The answer is yes.

This is a good moment to review the definition of~PA.
The full definition is somewhat complicated,
and is available in any number of textbooks on mathematical logic,
so we limit ourselves to a sketch.

The first thing to be aware of
is that even \emph{stating} the axioms of~PA requires
describing a \emph{formal language.}
Formulas in the
\emph{first-order language of arithmetic} are strings of symbols
satisfying certain syntactic rules.  There are logical symbols
$\vee$, $\wedge$, $\neg$, $=$, $\forall$, $\exists$
(the logical symbol $\implies$ can also be used, although it
is redundant since $P \implies Q$ is equivalent to $\neg P \vee Q$).
There are arithmetical function symbols $+$, $\times$, $0$, S,
and there is a relation symbol~$>$.
There are parentheses, used for grouping, and there are variables.
The syntactic rules allow us to write formulas such as
\begin{equation}
\label{eq:prime}
(z > \mbox{S0}) \wedge \forall x \forall y
((\neg (x\times y = z) \vee (x = \mbox{S0}) \vee (y = \mbox{S0}))
\end{equation}
Formula~\eqref{eq:prime} has two \emph{bound} variables
$x$ and~$y$, meaning that there is a quantifier attached to them,
and one \emph{free} variable~$z$.  Formulas with no free variables
are called \emph{sentences}.

The \emph{axioms} of~PA are certain sentences,
which are mostly what you would expect, e.g., formulas
for commutativity, associativity, distributivity, etc.
There is one axiom (or more precisely, an axiom
\emph{schema}, meaning a family of axioms satisfying a
certain ``template'') that is more subtle, namely
the \emph{induction axiom.}  Intuitively speaking, the induction
axiom says that if $P$ is a property that a natural number
might have, and if $0$ has~$P$, and moreover if
whenever $z$ has~$P$ then the successor of $z$ also has~$P$, then
every natural number has~$P$.  But what is a ``property''?

In \emph{first-order}
Peano arithmetic, which is the subject of the present article,
the induction axiom is asserted only for properties that
are \emph{expressible with a first-order formula}\footnote{There
is another version of the Peano axioms, usually known as the
\emph{second-order} Peano axioms, with the property that
there is only one mathematical structure satisfying them
(namely $\mathbb{N}$), and that can be used as a
\emph{definition} of~$\mathbb{N}$.  In contrast,
there are many non-isomorphic structures, known as
\emph{nonstandard models}, that satisfy the first-order
Peano axioms.}.
More precisely, for every first-order formula~$\phi(x, \mathbf{y})$
with free variables~$x, \mathbf{y}$ (here $\mathbf{y}$ represents
a finite sequence of variables),
we have an instance of the induction axiom
that looks something like this:
\begin{equation}
\label{eq:induction}
\forall \mathbf{y}
((\phi(0,\mathbf{y}) \wedge
\forall x (\phi(x,\mathbf{y}) \implies
\phi(Sx, \mathbf{y})))
\implies \forall x \phi(x, \mathbf{y}))
\end{equation}
\relax From the axioms, one can derive \emph{theorems} by
applying the \emph{rules of inference}
of first-order logic,
which are syntactic rules for manipulating formulas;
again, these rules are described
in textbooks and we will not enumerate them here.
We just remark that PA uses \emph{classical} rather
than \emph{intuitionistic} logic\footnote{As far as
the consistency of first-order arithmetic is concerned,
the distinction between intuitionistic logic and classical
logic turns out not to matter too much.  G\"odel,
and independently Gentzen~\cite{gentzen}, showed constructively that
\emph{Heyting arithmetic}, which is the intuitionistic
counterpart of PA, is consistent if and only PA is consistent.},
meaning that the rules include the \emph{law of the excluded middle}
(that allows one to deduce $P\vee\neg P$ for any~$P$).

Saying that \emph{PA is consistent} just means that
a contradiction---meaning a formula such as $(0=0) \wedge \neg(0=0)$
that is the conjunction of a formula and its negation---cannot
be derived from the axioms using the rules of inference.
Equivalently, since first-order logic is \emph{explosive,}
meaning that from a contradiction one can derive any (syntactically
well-formed) sentence whatsoever, \emph{PA is consistent} means
that there is some statement that is \emph{not} a theorem of~PA.

So far, our description of~PA has been purely \emph{syntactic}
and not \emph{semantic}.  That is, we have not assigned any
\emph{meaning} to the symbols.  \emph{Model theory} is the study
of mathematical structures that \emph{satisfy} given axioms;
to do model theory, we have to \emph{interpret} the symbols
$\vee$, $\wedge$, $\neg$, $=$, $\forall$, $\exists$ as
(respectively)
or, and, not, equals, for all, and there exists;
we also let the variables range over the elements\footnote{Note
in particular that the variables are \emph{not} allowed to
range over \emph{sets} of elements;
this restriction is what makes PA a \emph{first-order} theory.}
of the structure~$X$ that is to satisfy the axioms,
and we interpret the function and relation symbols 
as functions and relations on~$X$.

The standard way to show that some set of axioms is consistent is
to exhibit a structure that satisfies all the axioms.
In the case of~PA, the obvious candidate is~$\mathbb{N}$,
the set of natural numbers, with $+$, $\times$, $0$, S, and $>$
interpreted as addition, multiplication, zero, successor, and greater than.
After all, $\mathbb{N}$ was the example that motivated
the axioms of~PA in the first place.
Indeed, arguing set-theoretically, it is straightforward
to construct the natural numbers,
show that they satisfy all the axioms of~PA, and
conclude that PA is consistent.
This argument is easily formalized in ZFC.

It is worth remarking that this set-theoretic proof
of the consistency of~PA
does more than just show that the concept of an unbounded
sequence $1, 2, 3, \ldots$ is coherent;
if that were all it showed then it would not show very much,
since even asking whether PA is consistent
presupposes that \emph{the definition of~PA is coherent},
and that definition already implicitly assumes
that it is meaningful to talk about (certain kinds of) unbounded sequences,
such as arbitrarily long strings of symbols.
The ZFC proof affirms that first-order formulas involving
\emph{arbitrarily long alternations of quantifiers} (for all $x_1$
there exists $x_2$ such that for all $x_3$ there exists $x_4\ldots$)
express meaningful properties of natural numbers.
This claim goes beyond what is needed to construct PA itself.

\section{Implications of the ZFC Proof}

Under most circumstances, the formalizability in ZFC
of a proof of a statement~$S$ is
enough to cause people to regard $S$ as
``not an open problem.''
In fact, the above set-theoretic argument 
for the consistency of~PA
can be carried out using much weaker axioms than~ZFC,
and from a conventional mathematical standpoint is
just as rigorous as proving that the axioms for an algebraically
closed field are consistent by exhibiting~$\mathbb{C}$ as an example,
or proving that the axioms for a Hilbert space are consistent by
exhibiting $L^2([0,1])$ as an example.
If we regard a mathematical statement as being
\emph{definitively established} once it has been
\emph{mathematically proved},
then the consistency of~PA has been definitively established.
Nevertheless, many people find the above proof of the
consistency of~PA unsatisfactory.
Why might that be?

We can partially answer this question
by recalling some history\footnote{For much more historical
context, I recommend the article by Kahle~\cite{kahle}.}.
Especially during the late 19th
century and early 20th century, many mathematicians were concerned with
whether mathematical reasoning was trustworthy.  The paradoxes
of set theory had demonstrated that incautious use of superficially
valid mathematical reasoning could lead to contradictions, so
naturally, mathematicians were eager to delimit exactly which
reasoning principles were trustworthy and which were not.
One option was to be extremely conservative, but this came at
the cost of rejecting many mathematical proofs that seemed
perfectly fine, and not everyone was willing to give those up.
A variety of systems of varying logical strength were proposed
for formalizing various subsets of mathematical knowledge,
and PA was one candidate for formalizing \emph{arithmetical} knowledge.

Because of this potential role as a foundation for
part of mathematics,
people did not look at the axioms of~PA
in quite the same way that they looked at
axioms for an ``ordinary mathematical structure''
such as a differentiable manifold or a Lie algebra.
Many felt that a consistency proof for~PA should be held to
a \emph{higher standard} of rigor than usual---an ``ordinary''
mathematical proof might not be good enough,
since the consistency proof was supposed to certify
(to skeptics who raised doubts about certain kinds of
mathematical arguments) that the system was ``safe.''

In this context, someone could object that the
set-theoretic proof
employs dubious \emph{reasoning about infinity.}
Being finite creatures, we cannot apprehend infinite objects
in the same way that we can apprehend finite objects, and if
we reason about infinite objects by analogy with finite objects,
we might be on logically shaky ground.
If we re-examine the set-theoretic proof of the consistency of~PA,
then we see that it amounts to an argument that
there cannot be a contradiction
in the axioms of~PA, because there is an object---specifically,
an \emph{infinite} object,
namely~$\mathbb{N}$---that satisfies all those axioms.
A contradiction in~PA would mean that $\mathbb{N}$
simultaneously \emph{has} a
(first-order definable) property and \emph{does not have}
that property---but this is nonsense because an object either
has a property or it doesn't.

If you, like most mathematicians, find $\mathbb{N}$
and its first-order properties to be perfectly clear,
then the set-theoretic proof should satisfy you that PA is consistent.
But some might be uneasy that the argument seems to
presuppose the \emph{reality of infinite sets}
(sometimes referred to as \emph{platonism}
about infinite sets\footnote{On the other hand,
some people, such as
Solomon Feferman~\cite{feferman}, explicitly reject platonism 
but nevertheless find
the argument that $\mathbb{N}$ satisfies all the axioms of~PA
to be completely convincing.}).
Voevodsky noted in his talk that first-order properties of the
natural numbers can be \emph{uncomputable.}  This means that if our
plan is to react to a purported proof of $P\wedge\neg P$ by checking
directly whether $P$ or $\neg P$ holds for the natural numbers,
then we might be out of luck---we might not be able to figure out, in a
finite amount of time, which of $P$ and~$\neg P$ really holds of
the natural numbers.  In the absence of such a decision procedure,
how confident can we really be that the natural numbers must either
have the property or not?  Maybe the alleged ``property''
is meaningless.

This line of thinking may lead us to wonder if
``PA is consistent'' can be proved without assuming, as ZFC does,
that infinite sets exist.
After all, ``PA is consistent'' is a statement about what happens when
a finite list of rules is applied to finite strings of symbols, and if
there is a proof of a contradiction, then it must materialize after a
finite number of applications of those rules, and only finitely many
axioms can enter the picture.  It therefore seems plausible that we
might be able to give a \emph{finitary} proof that PA is consistent.
The word \emph{finitary} has no universally agreed-upon precise
definition, but following custom, we will use it informally to mean
methods of mathematical proof that try to avoid, or minimize,
assumptions about infinite quantities and processes.

\section{But Wait, What About G\"odel?}

At this point the reader might recall that G\"odel's Second
Incompleteness Theorem tells us that if PA is consistent, then
the consistency of~PA---or more precisely, a certain string
Con(PA) that ``expresses'' the consistency of PA---is not
provable in~PA.  Doesn't this theorem tell us that we cannot
hope to prove the consistency of~PA except by using an axiomatic
system that is \emph{stronger} than~PA?  And if that is the case,
then it would seem that we can \emph{never} be sure that PA is
consistent; if we have doubts about~PA, then any ``proof'' that
PA is consistent must rely on even more doubtful assumptions.
Any consistency proof must be circular
in the sense of assuming more than it proves,
so not only is the consistency of~PA an open problem,
it is doomed to remain open forever.

The above argument is correct, up to a point.
The MathOverflow question ``Is PA consistent? do we know it?''
asks more specifically whether the consistency of~PA has been
proved in ``a system that has itself been proven consistent.''
This question tacitly assumes that it is somehow possible to
``pull yourself up by your own bootstraps'' by setting up some
system whose consistency is \emph{guaranteed} because it has
been \emph{proven}---presumably in some absolute, unconditional
sense.  But any consistency proof has to assume \emph{something},
and you can always cast doubt on that ``something'' and demand
that \emph{it} be given a consistency proof, and so on ad infinitum.
Even if somehow you found a plausible system that
proved its own consistency\footnote{See
for example Willard~\cite{willard}
for an explanation of how this might be possible.},
any doubts you had about its consistency would hardly
be allayed just because it vouched for itself!
At some point, you simply have to take something for granted
without demanding that it be proved from something more basic.
This much is obvious, even without G\"odel's theorem.

Where the above argument goes wrong is the claim of circularity.
G\"odel's theorem does not actually say
that the consistency of~PA cannot be proved except in a system
that is \emph{stronger} than~PA.  It \emph{does} say that
Con(PA) cannot be proved in a system that is
\emph{weaker} than~PA, in the sense of a system whose theorems are a
subset of the theorems of~PA---and therefore Hilbert's original
program of proving statements such as Con(PA) or Con(ZFC)
in a strictly weaker system such as~PRA is doomed.
However, the possibility remains
open that one could prove Con(PA)
in a system that is \emph{neither weaker nor stronger} than~PA,
e.g., PRA together with an axiom (or axioms) that cannot
be proved in~PA but that we can examine on an individual basis,
and whose legitimacy we can accept.  This is exactly what Gerhard Gentzen
accomplished back in the 1930s, and it is to Gentzen's proof
that we turn next.

\section{Ordinals Below $\epsilon_0$}

The crux of Gentzen's consistency proof is something known
as the \emph{ordinal number~$\epsilon_0$}.
Some accounts of $\epsilon_0$ make it seem
``even more infinitary'' than the set of all natural numbers,
and so Gentzen's proof might seem to be even less satisfactory
than the ZFC proof, as far as suspicious axioms are concerned.
Therefore, this section gives a self-contained description of~$\epsilon_0$
that is as finitary as possible.
Our account borrows heavily from that of Franz\'en~\cite{franzen}.

I should remark that the discussion in this section and the next
is conducted using ``ordinary mathematics,'' and I advise
readers to use their ordinary mathematical ability
to digest the arguments, without at first worrying about
what assumptions are used in them.
The more subtle question of the minimal assumptions
needed for the proof can be addressed
after the arguments are understood.
We return to this question in Section~\ref{sec:gentzenimplications}.

Define a \emph{list} to be either an empty sequence
(denoted by [] and referred to as the \emph{empty list}),
or, recursively, a finite non-empty sequence of
lists.  So for example [[],[],[]] and [[[],[]],[[[],[[],[]]],[]]]
are lists.
The number of constituent
lists is called the \emph{length} of~$a$ (and is zero for
the empty list).
If $a$ is a non-empty list then we write $a[i]$ for
the $i$th constituent list of~$a$, where $i$ ranges from $1$
up to the length of~$a$.

Next, recursively define a total ordering~$\le$ on lists as follows
(it is essentially a lexicographic ordering).
Let $a$ and~$b$ be lists, with lengths $m$ and~$n$ respectively.
If $m\le n$ and $a[i] = b[i]$ for all $1\le i\le m$
(this condition is vacuously satisfied if $m=0$)
then $a\le b$.  Otherwise, there exists some $i$
such that $a[i] \ne b[i]$; let $i_0$ be the least such number,
and declare $a \le b$ if $a[i_0] \le b[i_0]$.

Finally, recursively define a list~$a$ to be an \emph{ordinal}
if all its constituent lists
are ordinals \emph{and} $a[i]\ge a[j]$ whenever $i<j$.
(In particular, the empty list is an ordinal since the
condition is vacuously satisfied.)

As an example, the smallest ordinals, listed in increasing order,
are [], [[]], [[],[]], [[],[],[]], [[],[],[],[]].
The ordinal [[[]]] is greater than all of these,
and [[[],[]]] is greater than [[[]],[[]]].

In the literature, what we here call an ordinal is called
the \emph{Cantor normal form} for an \emph{ordinal below~$\epsilon_0$},
and the standard notations for [[[]]] and
[[[],[]]] and [[[]],[[]]] are $\omega$ and $\omega^2$ and $\omega\cdot 2$
respectively.
We have chosen our notation to emphasize
that at no point are we appealing to the notion of
an infinite set.  Of course, the length of a list,
while finite, is unbounded, and hence if you wanted to
talk about \emph{the set of all lists} or
\emph{the set of all ordinals} then you would have to talk
about an infinite set.  However, there is no need to appeal
to such entities to make sense of our definitions.

The basic fact about ordinals is the following theorem.

\begin{theorem}
\label{thm:e0}
If $a_1, a_2, a_3, \ldots$ is a sequence of ordinals
and $a_i \ge a_j$ whenever $i<j$, then the sequence
stabilizes; i.e., there exists $i_0\ge 1$ such that
$a_i = a_{i_0}$ for all $i\ge i_0$.
\end{theorem}

The alert reader will notice that the statement of
Theorem~\ref{thm:e0} presupposes the concept of
an \emph{arbitrary infinite sequence} and hence is not finitary.
We will return to this point below, but first let us prove
Theorem~\ref{thm:e0}.
I encourage the reader to study the proof carefully,
since our later discussion about the correctness of Gentzen's proof
will be hard to appreciate otherwise.

\begin{proof}[Proof of Theorem~\ref{thm:e0}]
Define the \emph{height} $h(a)$ of an ordinal~$a$
to be the number of [ symbols at the beginning prior to
the first ] symbol.
It is easily proved by induction that
if $h(a) > h(b)$ then $a > b$;
equivalently, if $a\le b$ then $h(a) \le h(b)$.
The proof of Theorem~\ref{thm:e0} proceeds by induction on
$H := \min_i \{h(a_i)\}$.
If $H=1$, then $a_i$ is the empty list for
some~$i$, and since no list is strictly less
than the empty list, the sequence must stabilize at that point.

Otherwise, let us form the sequence $b_i := a_i[1]$.
Note that $h(b_i) = h(a_i)-1$.
Since $a_i \ge a_j$ whenever $i<j$, it follows that $b_i \ge b_j$
whenever $i<j$.
Therefore by induction,
all but finitely many of the $b_i$
are equal to a specific ordinal, which we call~$b$.
If we restrict attention to the $a_i$ such that $a_i[1]=b$,
then each such $a_i$ starts with some finite number of
repetitions of~$b$; let $B$ denote the smallest number of
repetitions (over all~$a_i$ such that $a_i[1] =b$).
Since the constituent lists of an ordinal are arranged
in weakly decreasing order, and since the $a_i$ are
arranged in weakly decreasing order, it follows that
the $a_i$ with exactly $B$ copies of~$b$
must come \emph{after} the $a_i$ with more than~$B$ copies of~$b$.
Hence all but finitely many of the~$a_i$
start with exactly $B$ copies of~$b$.
If some $a_i$ consists of exactly $B$ copies of~$b$ and nothing else,
then the sequence must stabilize at that point, and we are done.

Otherwise, restrict attention to those $a_i$ that start
with exactly $B$ copies of~$b$, and form the sequence
$c_i := a_i[B+1]$.  Then $h(c_i) \le h(b)$, so we can
repeat the same argument that we gave in the previous paragraph
to conclude that all but finitely many of the~$a_i$ start with
exactly $B$ copies of~$b$ followed by exactly $C$ copies of~$c$,
for some natural number~$C$ and some ordinal~$c < b$.
In this way we can inductively construct a decreasing sequence
of ordinals $b > c > d > \cdots$ of height less than~$H$.
Applying the induction hypothesis, this sequence must stabilize;
if it stabilizes with, say, $Z$ copies of~$z$,
then there must be some $a_j$ that precisely consists of $B$ copies
of~$b$ followed by $C$ copies of~$c$, etc., and terminating with
$Z$ copies of~$z$.  This $a_j$ must be $\le a_i$ for all~$i$,
and hence the sequence must stabilize with~$a_j$.
\end{proof}

We have stated and proved Theorem~\ref{thm:e0} in terms of
arbitrary infinite sequences,
because that is the easiest way to see what is going on.
For Gentzen's proof, though,
the following weak corollary of Theorem~\ref{thm:e0} suffices.

\begin{theorem}
\label{thm:e02}
If $M$ is a Turing machine\footnote{In fact, the
theorem can be further weakened to assert the stabilization
of all \emph{primitive recursive} descending sequences
of ordinals; see \cite[Lemma 12.79]{takeuti}
or \cite[Theorem 4.6]{buchholz} for example.
The fact that
PRA plus Theorem~\ref{thm:e02} implies that PA is consistent
is only implicit and not explicit in Gentzen's original proof.
I have chosen this way of presenting the argument
rather than the more common approach of explaining what
``induction up to~$\epsilon_0$'' is,
because I believe that Theorem~\ref{thm:e02}
is more accessible to the general
reader without training in logic and set theory.}
that, given $i$ as input,
outputs an ordinal~$M(i)$, and $M(i)\ge M(i+1)$ for all~$i$,
then the sequence stabilizes.
\end{theorem}

Although ordinals are commonly defined in the literature
using set theory, Theorem~\ref{thm:e02} can be formalized without
set theory; it can for example be phrased in the first-order
language of arithmetic, by using standard tricks for encoding
Turing machines and finite sequences using natural numbers.
In fact, Theorem~\ref{thm:e02} can \emph{almost} be proved in~PA.
The full justification of this claim is rather technical,
so again we will just sketch the idea.

First, we can formulate a theorem---call it Theorem~\ref{thm:e0}$'$---that
is intermediate in strength between Theorem~\ref{thm:e0}
and Theorem~\ref{thm:e02}, which restricts Theorem~\ref{thm:e0} to
weakly decreasing sequences of ordinals that are \emph{definable
by a first-order formula~$\phi$}.
To prove this version of the theorem,
suppose we have a formula~$\phi$
that defines a weakly decreasing sequence of ordinals
and asserts that they all have height at least~$H$.
Then we can mimic the proof of Theorem~\ref{thm:e0} to construct
a PA proof of Theorem~\ref{thm:e0}$'$ for~$\phi$.
The only catch is that we need, as building blocks,
PA proofs of Theorem~\ref{thm:e0}$'$ for formulas with smaller~$H$---but
we can assume by induction that these are available.
Note that this is an inductive procedure for \emph{constructing PA proofs
of individual instances of Theorem~\ref{thm:e0}$'$}
and cannot be converted to a PA proof of Theorem~\ref{thm:e0}$'$ itself;
however, it illustrates that each instance of Theorem~\ref{thm:e0}$'$
can be proved without assuming the existence of infinite sets.

\section{Gentzen's Consistency Proof}

Gentzen is usually regarded as having produced four different
versions of his consistency proof.  Only three versions were
published during his lifetime, but the first published version
is usually called his second proof, because it involved a major
revision of the version that he originally submitted for publication.
All versions of his proof may be found
in his collected works~\cite{gentzen}.
For our present purposes, the differences between the versions
are not critical, so we simply refer to ``Gentzen's proof''
without specifying the version.

Giving a full account of Gentzen's proof is beyond the scope of
this article because it necessarily involves careful attention to
the nitty-gritty details of~PA, but we give a sketch of the main
idea, following the account of Tait~\cite{tait}.
It is convenient to assume
that negation $\neg$ occurs only in \emph{atomic} formulas,
meaning those not involving $\vee$, $\wedge$, $\forall$, or $\exists$
(this can always be achieved because $\neg$ can always be ``pushed
inside'' at the cost of toggling between $\vee$ and $\wedge$
and between $\forall$ and $\exists$).  Imagine that you are playing
a game against an adversary, and the state of the board at any time
consists of a finite number of sentences.
Your goal is to reach a state in which one
of the sentences is a true atomic sentence.

The \emph{components} of the sentences $\phi \vee \psi$ and
$\phi \wedge \psi$ are $\phi$ and~$\psi$.  The components of
the sentences $\forall x\phi(x)$ and $\exists x\phi(x)$
are the sentences $\phi(\mbox{SSS}\cdots \mbox{S0})$ for some finite number
of occurrences of~$S$.  When it is your turn, you point to
a sentence~$\phi$, and if it is a $\vee$-sentence or an $\exists$-sentence,
then you add one of the components of~$\phi$ to the board,
and then you go again.
If you point to a $\wedge$-sentence or a $\forall$-sentence,
then it is your adversary's turn; the adversary adds a component
of~$\phi$ to the board and removes $\phi$ from the board,
and then it is your turn again.

To understand the point of the game, let us provisionally accept
the reality of~$\mathbb{N}$, and regard
sentences as making assertions about~$\mathbb{N}$
that are either true or false.  We are trying to show that
at least one of the sentences on the board is true by
instantiating all the variables with specific numbers
and reducing everything to an atomic sentence whose truth
can be directly checked by numerical calculation.  When a
universal quantifier shows up, we allow an adversary to instantiate
the variable since we are supposed to be able to win no matter
what the adversary picks.  Intuitively, we will have a winning
strategy---which, following Gentzen, we call a \emph{reduction}
of the initial state---if and only if at least one of the sentences
on the board is true.

If we now don our skeptical face and claim not to understand
what \emph{truth} means, we can forget about truth and simply
use the existence of a reduction as a surrogate for truth.
Now suppose we have a set~$\Gamma$ of sentences arising in
a formal PA-proof.
The core of Gentzen's proof, where the hard work is done, is
to construct, in an effective manner, a reduction of~$\Gamma$.
This is done inductively, by showing that if we have a reduction
of~$\Gamma$, and we introduce an axiom or a rule of inference,
then the resulting $\Gamma'$ also has a reduction.  The punchline
is that some sentences, such as $0=\mbox{S0}$, manifestly have no
reduction, and so are not derivable in~PA.

The reason ordinals show up in the proof is that they are used
to track game trees.  In particular, we need to be able to show
that reductions always terminate.  Proving this requires
Theorem~\ref{thm:e02} (or something similar).

\section{Implications of Gentzen's Proof}
\label{sec:gentzenimplications}

Gentzen's proof certainly meets ordinary standards
of mathematical rigor, but remember that we are trying
to adhere to higher than usual standards.
So what assumptions are really needed to carry out the proof?
Answering this question requires not just understanding the argument,
but also some experience with \emph{formalizing}
mathematical arguments.  Fortunately for us,
logicians have carefully analyzed the argument,
and the verdict is that other than Theorem~\ref{thm:e02},
everything in Gentzen's proof can be formalized in PRA,
which as we said earlier is a system of axioms
that is widely regarded as being finitary and very conservative.
In particular, PRA makes no reference to infinite sets.
Thus, Gentzen has reduced the
analysis of arbitrarily complicated first-order sentences of~PA,
and their classical logical consequences,
to a single finitary statement, namely Theorem~\ref{thm:e02}.
What objection might one have to Theorem~\ref{thm:e02}?

Voevodsky's objection was that Gentzen's only justification for
Theorem~\ref{thm:e02} was that it was self-evident---a suspicious claim,
according to Voevodsky, since G\"odel's theorem tells
us that Theorem~\ref{thm:e02} cannot be proved using ``usual
induction techniques.''
If we take this objection at face value,
then it is at best misleadingly phrased.
Gentzen does not say that Theorem~\ref{thm:e02} (or rather,
the variant of it that he uses in his proof) is self-evident;
he gives an inductive argument along the lines we have given.
As we have seen, by normal mathematical standards,
there is nothing particuarly ``unusual'' about the inductive
argument\footnote{A far stronger induction
argument was used by Robertson and Seymour in their proof of
the Graph Minor Theorem~\cite{frs},
and nobody seems to have rejected the
Graph Minor Theorem on those grounds.}.
The only way I have been able to make sense of Voevodsky's argument
is by interpreting him as assuming
that a consistency proof for a system can be convincing
only if it can be carried out in a system strictly weaker than the
system itself.
If we accept this assumption then we can indeed view
G\"odel's theorem as a dealbreaker,
but then Voevodsky's objection becomes a blanket rejection
of \emph{all} consistency proofs,
and has nothing to do with any specific concerns about PA
or Gentzen's proof.
As we argued earlier, G\"odel's theorem, which Voevodsky cites in support
of his objection, does not entail such blanket skepticism.

It could be that Voevodsky's real concern
was that even though the \emph{statement}
of Theorem~\ref{thm:e02} is finitary,
it does not feel like an axiom,
and the only ways to justify it seem to be infinitary.
Gentzen tried to argue that the induction needed for his proof
was just more complicated, and not different in character, from
the finitary induction argument that any weakly decreasing
sequence of \emph{natural numbers} must eventually stabilize.
But since ``finitary'' is not precisely defined,
the point can be legitimately debated.
Note, though, that rejecting the proof of Theorem~\ref{thm:e0}
comes with a cost;
it potentially means that many routine mathematical arguments
by induction are suspect---not just those involving
arbitrarily complex first-order properties.

Alternatively, Voevodsky's real concern may have been that
the proof of Theorem~\ref{thm:e02} is insufficiently \emph{constructive},
since the stabilization point is not, in general, computable.
Again, this could be a tenable objection,
but it comes at a price, because
rejecting all ``uncomputable mathematics'' means rejecting
a sizable fraction of all mathematics.
A plausible candidate
for an axiomatization of ``computable mathematics''
(assuming classical logic and not intuitionistic logic)
is a system known as RCA$_0$~\cite{simpson}.
In RCA$_0$ one cannot prove the consistency of~PA,
but one cannot prove Brouwer's fixed-point theorem
or the Bolzano--Weierstrass theorem either.

\section{Friedman's Relative Consistency Proof}

Speaking of the Bolzano--Weierstrass theorem,
we should mention a result due to Harvey Friedman,
announced on the Foundations of Mathematics
mailing list~\cite{friedmanfom}
but not formally published, that the inconsistency
of~PA would imply the inconsistency of a system called SRM+BWQ.
Here SRM (Strict Reverse Mathematics~\cite{friedmansrm})
is a weak system of axioms that serves as a ``base theory,''
and BWQ (Bolzano--Weierstrass for~$\mathbb{Q}$)
is the familiar mathematical principle that every
bounded infinite sequence of rationals has an infinite Cauchy
subsequence.

Friedman's proof is not directed at those who are skeptical of
infinite sets or uncomputable sequences,
since it uses both concepts
(the set of indices of the subsequence promised by BWQ can,
and usually will, be an uncomputable set of natural numbers,
even if the original sequence is computable).
Rather, it is directed at those who feel that
formal systems for mathematics are artificially strong and
overly general, and who argue that ``natural'' mathematical
statements require only a limited set of induction principles.
In particular, they reject the inductive proof
of Theorem~\ref{thm:e0} as being unnaturally strong.
Friedman argues that SRM+BWQ uses only principles
that are routinely accepted in ``mainstream mathematics,''
and hence that anyone who accepts that ordinary mathematical
reasoning is consistent should accept that PA is consistent.

Even a sketch of Friedman's proof requires concepts that go
beyond the scope of this article, but since his argument is
not well known, we say a few words here for the benefit
of readers with some background in logic.
If we replace SRM with RCA$_0$,
then the result is proved in Simpson~\cite[Theorem I.9.1]{simpson}.
The key point is that the unbounded existential quantifier
in BWQ allows one to construct computably enumerable sets
(e.g., the set of all Turing machines that halt) from
computable approximations.
In the terminology of second-order arithmetic,
this lets us pass from $\Delta_1^0$ comprehension to
$\Sigma_1^0$ comprehension,
which can then be ``bootstrapped'' up to arithmetical comprehension.
Therefore every axiom of PA can be derived in RCA$_0$+BWQ,
yielding a relative consistency proof.
In SRM, one strips down this argument to its bare essentials
to avoid ``unnecessary generality,'' but BWQ still plays the
same role of providing the crucial unbounded existential quantifier.
Note that a variety of other mathematical statements besides BWQ
could do the job equally well.

\section{Taking Stock}

There are other ways to prove the consistency of~PA
(e.g., there is a relative consistency proof based on G\"odel's
``Dialectica'' interpretation~\cite{avigad-feferman}),
but the results we have discussed so far already show
that the normal mathematical standards for declaring
something to be proved, known, solved, and no longer an open problem
have been met and even exceeded.
Even those who are doubtful about \emph{some}
mathematical methods may still be able to regard
the consistency of~PA as being settled.
\begin{enumerate}
\item
If we believe that $\mathbb{N}$ must either have, or not have,
every property expressible in the first-order language of arithmetic,
then the straightforward set-theoretic proof should
satisfy us that PA is consistent.
\item
If we are doubtful about the meaningfulness of arbitrary
first-order properties of~$\mathbb{N}$, but we
believe Theorem~\ref{thm:e02}, along with routine mathematical principles
that are much simpler than Theorem~\ref{thm:e02}, then Gentzen's
proof should satisfy us that PA is consistent.
\item
If we believe that SRM+BWQ is consistent, then
Friedman's proof should convince us that PA is consistent.
\end{enumerate}
On the other hand, if we are exceptionally cautious,
we might reject all these proofs as using unjustified principles---but
if we do so, then we will have to reject significant portions of
ordinary mathematics as being unjustified as well.

Our discussion could end here, but some readers
may still be uneasy with the reference to
\emph{belief} (in infinite sets or Theorem~\ref{thm:e02} or~BWQ),
and the introduction of shades of gray into
a discussion about mathematics.
Isn't the point of mathematics to eliminate the need
for philosophical mumbo-jumbo and subjective, mystical beliefs,
and to rely on \emph{proof} instead?

The desire to avoid (or at least minimize) philosophical
assumptions, and defend the objectivity of mathematics,
leads some people to a point of view known as
\emph{formalism}.
Edward Nelson in particular
was a self-avowed formalist~\cite{nelsonformalism},
and even refused to believe in PRA.
Can formalism save us from having to make a personal
decision about what to believe?

\section{The Formalist Perspective}

The term \emph{formalist} has no mathematically precise definition.
Its meaning has changed slightly over time,
and different people mean different things by it.
I will give a description that I believe captures the main idea.

The formalist regards mathematics as a formal game played with
symbols.  There are rules for how the symbols are allowed to be
manipulated.  Importantly, the symbols \emph{have no meaning.}
If we say that \emph{every differentiable function is continuous,}
it does not mean that there really are such things as functions,
and that differentiability and continuity are real properties that
functions really have, and that every function that has the
differentiability property also has the continuity property.
Rather, all we are saying, in an abbreviated shorthand, is that
``every differentiable function is continuous'' \emph{is a theorem
of~ZFC} (or perhaps a theorem of some other axiomatic system that we 
are interested in).

For the formalist,
the \emph{only} meaningful mathematical statements we can make are
syntactic statements about strings of symbols.
It is also common, though not universal,
for formalists to say that even statements about syntactic objects are
meaningful only when they are short enough for us to apprehend and
manipulate physically.
That is, formalists are often \emph{ultrafinitists}.
For an ultrafinitist, even a statement such as
``$2^{77232917}-1$ is prime'' does not, as one might na\"\i vely
think, mean that if (for example) we took $2^{77232917}-1$ marbles
and tried to arrange them in a rectangular pattern, then the only
way to do so would be to arrange them in a straight line.
The problem is that we cannot possibly lay our hands on
$2^{77232917}-1$ marbles, so what ``$2^{77232917}-1$ is prime''
means is just that we have verified that our rules for
manipulating symbols such as ``$2^{77232917}-1$'' have produced
a certain result.
Formalists thus not only reject the reality of infinite sets,
but they often reject the reality of natural numbers
as well.  They may say that they do not know what it means to
say that ``there exists a prime number between 50 and~100'' other
than that this statement is a theorem of some formal system.

One of the selling points of formalism is that it allows us to
sidestep questions about whether infinite sets exist,
or even whether we believe this axiom or that axiom.
Is the continuum hypothesis true or false?
The formalist says, ask not whether the continuum hypothesis
is true or false; ask only whether it has been proved in this
system or that system.
Whereof one cannot speak thereof one must be silent.
The formalist thus seems to offer us a way to salvage the
objectivity of mathematics in the face of competing
axiomatic systems.  If you have a private mystical belief
in infinite sets, that's your business, says the formalist,
but in the mathematical marketplace, the only legal tender
is mathematical proof---the \emph{deduction} of theorems from
axioms, and not any questions about the \emph{truth} of the
axioms.

What does this mean about the consistency of~PA?
At first glance, it seems that the formalist approach should
be to sidestep the question of whether PA is \emph{really}
consistent.  Ask not whether PA is really consistent; ask only
whether ``PA is consistent'' is provable in this or that system.
It is provable in ZFC; it is provable in primitive recursive
arithmetic plus Theorem~\ref{thm:e02}; end of story.

Unfortunately, the matter is not quite so simple,
and formalists do not react this way.
The issue is this: ``PA is inconsistent'' states that
manipulating certain symbols according to certain rules
will produce a certain result, and this is precisely the
sort of statement that even a formalist agrees is directly
meaningful---at least if the length of the proof is sufficiently
short.  Therefore a formalist cannot dodge the question of
whether PA is consistent or inconsistent.

If a formalist must confront the consistency question,
then in the absence of an explicit derivation of a contradiction from
the axioms of~PA, what kinds of arguments might a formalist
accept as establishing that PA is consistent?

Different formalists might have different answers to this question,
but I would like to argue that for at least
one flavor of formalist---which I
will dub a \emph{strict formalist}---the answer is that
\emph{no mathematical argument can definitively establish
the consistency of~PA.}
Hence, if PA is in fact consistent, its consistency will remain,
for the strict formalist, an ``open problem'' \emph{permanently.}

What do I mean by a \emph{strict formalist}?
A strict formalist---let's call him Stefan---is able to recognize,
and verify as correct,
any existing formal mathematical proof, by following the syntactic rules.
But Stefan takes very seriously the statement
that \emph{symbols have no meaning}.
Just as symbols cannot be construed as ``referring'' to
manifolds or functions or integers, \emph{symbols cannot be
construed as referring to syntactic entities either}.
Any mathematical argument
that purports to \emph{prove} that PA is consistent is really just
a finite derivation of the meaningless string Con(PA)
from some other strings.
Stefan can \emph{manipulate} syntactic objects
but cannot interpret a mathematical proof
as saying anything \emph{about} syntactic objects.
Even if Stefan discovers a contradiction in~PA
and exclaims, ``PA is inconsistent!''  he will not
identify this meaningful English statement with the meaningless
string $\neg$Con(PA).

Stefan avoids all accusations of accepting
``PA is consistent'' for unfounded, mystical reasons, but
at the cost of throwing out the baby with the bathwater---Stefan
also cannot accept
most of what passes for mathematical knowledge.
For example, suppose we design a computer program
to search for positive natural numbers $a$ and~$b$ such that
$a^2=2b^2$.  Stefan has no \emph{conclusive} grounds
for believing that such a search is futile.  Granted, just as physicists
strongly believe certain well-confirmed physical theories, such as
the seeming impossibility of transmitting information faster than
the speed of light, Stefan may agree that it is a
``well-confirmed mathematical theory'' that our program
will never find what it is looking for.
However, the conviction
that conventional mathematicians have, that the \emph{proof} of
the irrationality of~$\sqrt 2$ gives us an \emph{a priori guarantee}
that the search will never terminate,
is unavailable to Stefan.

Stefan is thus faced with a puzzle that I call the
\emph{unreasonable soundness of mathematics.}
Stefan can observe that
conventional mathematicians are remarkably successful at
making accurate predictions of the results of syntactic
manipulations, but has no explanation for this success\footnote{Note
that the unreasonable soundness of mathematics
is not the same as Eugene Wigner's unreasonable effectiveness
of mathematics in the natural sciences.
What Stefan cannot explain are mathematicians' purely \emph{mathematical}
predictions rather than their \emph{scientific} predictions.}.

In practice, I suspect that few if any mathematicians are
strict formalists.  (Nelson was not, since he believed that
``demonstrably consistent'' formal systems
were possible~\cite{nelsonformalism}.)
Part of the reason may be that even though formalists
often pride themselves on their rejection of the reality of
abstract objects such as natural numbers,
they \emph{do} accept the \emph{reality of symbols}
and the \emph{reality of syntactic rules},
and these concepts are very close to natural numbers
and arithmetical operations on natural numbers.
Note that a \emph{symbol is an abstract entity}.
I can pick up a piece of chalk
and write ``$\phi$'' on a blackboard, and point to it,
but the symbol ``$\phi$'' is not identical to the collection
of chalk particles on the blackboard.
I could have written $\phi$ on a piece
of paper, or I could have typed \texttt{\char92 phi} into a computer and
used \TeX\ to convert it to pixels on a screen,
and if all these multifarious physical entities are
supposed to be \emph{the same symbol},
then a ``symbol'' must be an abstract entity.
Moreover, in order to distinguish SSSS0 from SSSSS0,
I have to be able to \emph{count},
and it is very fine line between affirming the objectivity of counting
and affirming the reality of small natural numbers.
For a human being,
it is a very short step from being able to follow
syntactic rules to \emph{reasoning} about the outcome,
and before you know it, you find yourself insisting that if
you start with the string ``0'' and
all you do is repeatedly apply
the rule ``prepend an S'' to it,
then you will \emph{never} get a string with
(say) a ``$\wedge$'' in it,
even though all Stefan is equipped to do
is verify the absence of a $\wedge$ from
the strings S0, SS0, SSS0, etc., on a case-by-case basis.

If someone abandons strict formalism and accepts that
at least some types of mathematical reasoning
can provide secure knowledge about syntactic objects,
then we are back to shades of gray---one simply has
to decide what mathematical principles one accepts,
and then, depending on how strong those principles are,
one may or may not be able to conclude that PA,
or some other axiomatic system, is consistent.

\section{Finite Approximations to Consistency}

There is an angle on the consistency question that
someone who is not quite a strict formalist but who
has ultrafinitist leanings---let's call her Ulphia---might take.
Namely, Ulphia might not consider
the conventional reading of ``PA is consistent'' to be meaningful.
Instead, Ulphia might regard as meaningful only
what a conventional mathematician would call a
\emph{finite approximation} to the consistency of~PA,
by which I mean something like the following:
\begin{equation}
\label{eq:approx}
\mbox{The shortest PA proof of a contradiction has length $>n$}
\end{equation}
where $n$ is some number of feasible size.
If Ulphia believes in \emph{some} reasoning principles,
then presumably a proof of~\eqref{eq:approx} using those principles
(with $n$ being near the upper limit of feasibility)
would convince her that
searching for a PA proof of a contradiction
would be a wild goose chase.

If we let Con(PA,$n$) denote the statement that there is no
PA proof of a contradiction of length less than~$n$,
then we can ask for the length of the shortest PA proof of Con(PA,$n$).
Friedman has proved an $n^\epsilon$ lower bound on this length
(for some $\epsilon>0$), and 
and Pudl\'ak has proved a polynomial upper bound.
More interesting philosophically is the length of
the shortest proof of Con(PA,$n$) in a weaker
system, such as PRA, or even weaker systems such as bounded arithmetic.
Unfortunately such questions hinge on notorious unproved conjectures
in complexity theory, so almost nothing is known unconditionally.
Pudl\'ak and others conjecture superpolynomial lower bounds;
these would imply that even if Ulphia accepts some such system~$S$,
then any proof~$P$ in~$S$ that you can show her will only rule out
PA proofs of a contradiction
that are \emph{much} shorter than~$P$ itself,
and so will not necessarily convince her that it is pointless to search
for PA proofs of a contradiction.
For more on this subject, see~\cite{pudlak} and the references therein.

\section{Concluding Remarks}

Mathematicians typically take the attitude
that mathematical statements are either settled or open,
known or not known, proved or not proved,
and that mathematics is completely objective
and relies on nothing that is unproven.
But what this attitude glosses over
is that accepting a proven theorem
requires accepting the assumptions on which the proof is based.
This simple principle applies not only to
theorems that go beyond ZFC, but to \emph{every} theorem.

We have seen that by the usual standards of mathematical rigor,
the consistency of PA is a proven theorem and not an open problem.
On the other hand, you are free to reject
``the usual standards'' in favor of some other, stricter standards.
Depending on what those standards are,
you may or may not be able to conclude that PA is consistent.
If you want to minimize the assumptions you make,
then you might gravitate towards formalism,
but doing so might mean giving up much if not all
of what is commonly regarded as rigorously established mathematics.
In mathematics, as in life, there is no free lunch.

Earlier, we raised the question of whether an
inconsistency in~PA would cause all of mathematics to
come crashing down like a house of cards.
Would we all be doomed to suffer the fate of
the protagonist in Ted Chiang's short story
``Division By Zero''~\cite{chiang},
who discovers a contradiction in mathematics and is
unable to cope?
If we regard mathematics as a monolithic entity
with only one possible foundation on which everything depends,
then the answer might seem to be yes,
but if we recognize that there is a sliding scale
of axiomatic systems ranging from very weak systems
all the way up to large cardinal axioms in set theory,
then the answer is no.
If PA were found to be inconsistent
then most likely we would simply analyze the inconsistency
and adopt some other axiomatic system that avoids the problem.
For example, there exist \emph{paraconsistent} logics~\cite{carnielli}
that are not explosive, and that can recover gracefully from a
contradiction.
There is also an entire field called
\emph{reverse mathematics}~\cite{simpson} devoted to
analyzing exactly which axioms are needed for which theorems---but
that is a topic for another essay.

\section{Acknowledgments}

I would like to thank Anton Freund for pointing me to
the work of Takeuti~\cite{takeuti} and Buchholz~\cite{buchholz},
and Harvey Friedman, Sam Moelius, Joe Shipman, Bill Tait, and Bill Taylor
for helpful comments on an earlier draft of this paper.

\end{document}